\documentclass[11pt]{amsart}
\usepackage {enumerate}
\usepackage{setspace}
\usepackage{caption}
\usepackage{mathrsfs}
\usepackage{hyperref}

\usepackage{amssymb}
\usepackage{comment}
\usepackage{color}
\usepackage{lipsum}
\usepackage{graphicx}
\usepackage{esint}
\usepackage{epstopdf}

\newtheorem{theorem}{Theorem}[section]
\newtheorem{lemma}[theorem]{Lemma}
\newtheorem{corollary}[theorem]{Corollary}
\newtheorem{proposition}[theorem]{Proposition}

\newtheorem{conjecture}[theorem]{Conjecture}
\newtheorem{question}[theorem]{Question}
\numberwithin{equation}{section}

\theoremstyle {definition}
\newtheorem{definition}[theorem]{Definition}

\DeclareMathOperator{\area}{area}

\DeclareMathOperator{\Lip}{Lip}
\DeclareMathOperator{\sys}{sys}
\DeclareMathOperator{\dist}{dist}
\DeclareMathOperator{\fillrad}{Fill-Rad}
\DeclareMathOperator{\width}{width}
\DeclareMathOperator{\length}{length}
\DeclareMathOperator{\tr}{tr}
\DeclareMathOperator{\Div}{div}

\usepackage{amsfonts}
\usepackage{amssymb}
\usepackage[utf8]{inputenc}

\usepackage[english ]{babel}

\usepackage{bbding}

\usepackage[all]{xy}

\usepackage{tikz-cd}
\usepackage[all]{xy}
\usepackage{amsfonts}
\usepackage{amssymb}
\usepackage{graphicx}
\usepackage{mathtools}
\usepackage{skak}
\usepackage{foekfont}

\usepackage{calligra}
\usepackage[T1]{fontenc}

\usepackage{amsmath,mathdots}

\usepackage{stmaryrd}

\usepackage[utf8]{inputenc}

\usepackage{mathtools}
\usepackage{mathdots}
\usepackage{tikz}

\definecolor{amber(sae/ece)}{rgb}{1.0, 0.49, 0.0}
\newfont{\rsfsten}{rsfs10 scaled 1200}
\usepackage{MnSymbol}
\usepackage{tikz}
\usepackage{amscd}
\usepackage{MnSymbol}
\usepackage{wasysym}

\usepackage{mathrsfs}

\usepackage{pdfcolmk}

\usepackage{graphicx}

\DeclareMathOperator{\Ric}{Ric}
\DeclareMathOperator{\vol}{Vol}

\makeatletter
\@namedef{subjclassname@2020}{%
  \textup{2020} Mathematics Subject Classification}
\makeatother

\begin{document}

\title{The Gauss-Bonnet inequality beyond aspherical conjecture}
\author{Jintian Zhu}
\address[Jintian Zhu]{Beijing International Center for Mathematical Research, Peking University, Beijing, 100871, P.~R.~China}
\email{zhujt@pku.edu.cn, zhujintian@bicmr.pku.edu.cn}
\date{\today}
\begin{abstract}
Up to dimension five, we can prove that given any closed Riemannian manifold with nonnegative scalar curvature, of which the universal covering has vanishing homology group $H_k$ for all $k\geq 3$,  either it is flat or it has Gauss-Bonnet quantity (defined by \eqref{Eq: Gauss Bonnet quantity}) no greater than $8\pi$. In the second case, the equality for Gauss-Bonnet quantity yields that the universal covering splits as the Riemannian product of a $2$-sphere with non-negative sectional curvature and the Euclidean space. We also establish a dominated version of this result and its application to homotopical $2$-systole estimate unifies the results from \cite{BBN2010} and \cite{Zhu2020}.
\end{abstract}
\maketitle
\section{Introduction}
\subsection{Gauss-Bonnet inequality}
In Riemannian geometry, the integral of scalar curvature on a closed surface $(\Sigma,g)$ is characterized by the following Gauss-Bonnet formula
\begin{equation}\label{Eq: GB formula}
\int_\Sigma R_\Sigma\,\mathrm d\sigma_g=4\pi\chi(\Sigma),
\end{equation}
where $R_\Sigma$ is the scalar curvature of $(\Sigma,g)$, $\mathrm d\sigma_g$ is the area element of $(\Sigma,g)$ and $\chi(\Sigma)$ is the Euler number of $\Sigma$.

In research of three dimensional closed manifolds, integral of ambient scalar curvature on certain surfaces is considered in many works and it builds a bridge between curvature and topology of the ambient manifold. With a rearrangement in the second variation formula, Schoen and Yau \cite{SY1979} show that a stable two-sided minimal surface $\Sigma$ in a Riemannian $3$-manifold $(M,g)$ satisfies the Gauss-Bonnet inequality
\begin{equation}\label{Eq: GB ineq surface}
\int_\Sigma R(g)\,\mathrm d\sigma_g\leq 4\pi\chi(\Sigma).
\end{equation}
They further use this Gauss-Bonnet inequality to rule out the existence of smooth metrics with positive scalar curvature on manifolds containing an incompressible surface with positive genus (refer to \cite{SY1979, SY1982}). As for geometrical application, the Gauss-Bonnet inequality \eqref{Eq: GB ineq surface} also yields the existence of small $2$-spheres in a Riemannian $3$-manifold $(M,g)$ with non-trivial $\pi_2$ and large positive scalar curvature (refer to \cite{BBN2010} for the precise statement).

In this paper, we would like to figure out whether inequality \eqref{Eq: GB ineq surface} holds in Riemannian manifolds with dimension higher than three for certain non-trivial surface $\Sigma$. The setting in this paper is as follows.
Denote $(M^n,g)$ to be a Riemannian manifold with non-trivial $\pi_2$ and we introduce the {\it Gauss-Bonnet quantity} to be
\begin{equation}\label{Eq: Gauss Bonnet quantity}
Q_{GB}(M,g)=\inf\left\{\left.\int_{\mathbf S^2}R(g)\circ f\,\mathrm d\sigma_{f^*g}\,\right|\,f:\mathbf S^2\to M,\,[f]\neq 0\in\pi_2(M)\right\}.
\end{equation}
Inspired from the Gauss-Bonnet inequality \eqref{Eq: GB ineq surface} it is natural to ask
\begin{question}\label{Q: GB ineq}
Do we have $Q_{GB}(M,g)\leq 8\pi$ for any closed Riemannian manifold $(M,g)$?
\end{question}
First notice that Question \ref{Q: GB ineq} above only makes sense when $(M,g)$ has nonnegative scalar curvature. Otherwise we would have $Q_{GB}(M,g)=-\infty$ and the desired inequality holds trivially. For this reason, we just limit our attention to those Riemannian manifolds with nonnegative scalar curvature in the following context.

Clearly the answer to Question \ref{Q: GB ineq} is yes in dimension two and three. It is an immediate corollary of Gauss-Bonnet formula in dimension two and the work \cite{SY1979} in dimension three. In general the answer may be negative for Riemannian manifolds with dimension greater than $3$. For example, if $(M,g)$ is isometric to $\mathbf S^2(1)\times \mathbf S^3(\epsilon)$, then the Gauss-Bonnet quantity is
$$Q_{GB}(M,g)=4\pi(2+6\epsilon^{-1}),$$ which can be arbitrarily large as $\epsilon\to 0$. Therefore,
additional conditions are necessary for the validity of the desired inequality $Q_{GB}(M,g)\leq 8\pi$. In recent work \cite{GZ2021}, the author and Gromov can establish a similar Gauss-Bonnet inequality as \eqref{Eq: GB ineq surface} when given Riemannian $n$-manifold $(M,g)$ has $(n-2)$ directions large. Audience of interest can refer to \cite{GZ2021} for more details.

Here we shall consider Question \ref{Q: GB ineq} above from another point of view.
The philosophy that {\it one can attribute properties of scalar curvature on closed manifolds to the homotopy type of its universal covering} arises from the aspherical conjecture, which appears as a corollary of the strong Novikov conjecture (see \cite{Rosenberg1983}). Indeed, the aspherical conjecture asserts that if a closed manifold has contractible universal covering, then it admits no smooth metric with positive scalar curvature. This conjecture is now known to hold for closed manifolds with dimension no greater than five in a more general dominated version (see \cite{GL1983,SY1987, Gromov2020, CL2020}). Along a very similar philosophy we would like to give an affirmative answer to Question \ref{Q: GB ineq} for those closed manifolds whose universal covering $\hat M$ has vanishing homology group $H_k(\hat M,\mathbf Z)$ for all $k\geq 3$. Namely we have

\begin{theorem}\label{Thm: main}
Let $2\leq n\leq 5$ be an integer. Assume that $(M^n,g)$ is a closed orientable Riemannian manifold with nonnegative scalar curvature, whose universal covering $\hat M$ has vanishing homology group $H_k(\hat M,\mathbf Z)$ for all $k\geq 3$. Then we have
\begin{itemize}
\item either $(M,g)$ is flat;
\item or we have
$
\pi_2(M)\neq 0$ and
$Q_{GB}(M,g)\leq 8\pi,
$
where the equality holds if and only if the universal covering of $(M,g)$ is isometric to $(\mathbf S^2\times \mathbf R^{n-2},g_1+g_{euc})$, where $g_1$ is a smooth metric on $\mathbf S^2$ with nonnegative sectional curvature and $g_{euc}$ is the Euclidean metric on $\mathbf R^{n-2}$.
\end{itemize}
\end{theorem}

After a slight modification of the proof for Theorem \ref{Thm: main}, we can strengthen Theorem \ref{Thm: main} to a dominated version. The precise statement is as below.

\begin{theorem}\label{Thm: main dominated}
Let $2\leq n\leq 5$ be an integer. The conclusion of Theorem \ref{Thm: main} also holds for any closed orientable Riemannian manifold $(M^n,g)$ with nonnegative scalar curvature, which admits a non-zero degree map to another closed orientable manifold $M_0$ whose universal covering $\hat M_0$ has vanishing homology group $H_k(\hat M_0,\mathbf Z)$ for all $k\geq 3$.
\end{theorem}

We point out that Theorem \ref{Thm: main dominated} can be viewed as a further refinement of the dominated version of aspherical conjecture proven in low dimensions (see \cite{CL2020}).
As an immediate consequence, Theorem \ref{Thm: main dominated} further indicates an optimal homotopical $2$-systole estimate. Recall from \cite{BBN2010} that the {\it homotopical $2$-systole} is defined to be
\begin{equation*}
 \sys_2( M, g):=\inf\left\{\area(\mathbf S^2,i^* g)\left|\begin{array}{c}
 \text{$i:\mathbf S^2\to M$ smooth such that}\\
 \text{$[i]$ homotopically nontrivial}
 \end{array}\right.\right\}.
\end{equation*}

Now it is not difficult to see

\begin{corollary}\label{Cor: systole}
Let $2\leq n\leq 5$ be an integer. If $(M^n,g)$ is a closed orientable Riemannian manifold with positive scalar curvature, which admits a non-zero degree map to an orientable closed manifold $N$ whose universal covering $\hat N$ has vanishing homology group $H_k(\hat N,\mathbf Z)$ for $k\geq 3$, then we have
$$
\pi_2(M)\neq 0 \mbox{ and }
\inf_{M}R(g)\cdot\sys_2(M,g)\leq 8\pi,
$$
where the equality holds if and only if the universal covering of $(M,g)$ is isometric to $\mathbf S^2(1)\times \mathbf R^{n-2}$ up to scaling.
\end{corollary}
In his Ph.D. thesis, the author raised up the following {\it homotopical $2$-systole conjecture}:
\begin{conjecture}
If $(M,g)$ is a closed Riemannian manifold with $R(g)\geq 2$, whose universal covering is homotopic to $2$-sphere, then we have $$\sys_2(M,g)\leq 4\pi.$$
\end{conjecture}
We conclude that Corollary \ref{Cor: systole} gives an affirmative answer to the above conjecture for dimension no greater than five. Besides this, it is also worth to emphasize that Corollary \ref{Cor: systole} unifies the homotopical $2$-systole estimates from those works in \cite{BBN2010} and \cite{Zhu2020}.
In \cite{BBN2010}, Bray, Brendle and Neves proved the desired homotopical $2$-systole estimate only assuming that the given closed $3$-manifold has non-trivial second homotopy group while the work in \cite{Zhu2020} needs to require the closed manifold admitting a non-zero map to $\mathbf S^2\times T^{n-2}$. At the first sight, these conditions are quite different but in both cases we can verify the given manifold satisfying the hypothesis of Corollary \ref{Cor: systole}.

Finally let us say some words on the proof of Theorem \ref{Thm: main}. The basic idea  is to search for a core $2$-sphere representing a non-trivial element in $\pi_2$ and also satisfying \eqref{Eq: GB ineq surface} when the manifold $(M,g)$ has positive scalar curvature. The method here is essentially based on the slice-and-dice argument from Chodosh and Li \cite{CL2020}, which provides us many spheres and disks with small integral of ambient scalar curvature on them. In the presence of positive scalar curvature and the $H_k$-vanishing condition of the universal covering for $k\geq 3$,
we can further apply a selection argument to these spheres and disks such that we can construct a homotopically non-trivial $2$-sphere from gluing those carefully chosen spherical and disc pieces, which finally satisfies \eqref{Eq: GB ineq surface} up to an arbitrarily small error. This fact combined with a deformation argument as well as the Cheeger-Gromoll splitting theorem \cite{CG1971}
can be used to show the desired Gauss-Bonnet inequality $Q_{GB}(M,g)\leq 8\pi$ as well as the desired rigidity.

The rest of this paper will be arranged as follows. Section 2 is devoted to a review on cut-and-paste and slice-and-dice argument in detail.  In Section 3, we introduce the selection argument and present proofs for Theorem \ref{Thm: main} and Theorem \ref{Thm: main dominated}.

\vspace{2mm}
{\it Acknowledgement.} The author would like to thank Professor Gang Tian for many inspiring conversations. He also thanks Professor Yuguang Shi for constant support. This work is supported by the
China postdoctoral science foundation (grant BX2021013). He also thanks the referee for their valuable suggestions and comments.

\section{Preparations}
In this section, we are going to review the cut-and-paste and slice-and-dice arguments, which are quite useful in the proof of aspherical conjecture for dimensions no greater than five. Even though many results are known, we need to modify them to more quantitive versions for later use and so we just include all detailed proofs here for completeness.
\subsection{Extrinsic cut-and-paste}
\begin{lemma}[Extrinsic cutting lemma]\label{Lem: extrinsic cutting}
Let $(\hat M^n,g)$ be a connected complete Riemannian manifold with $H_{n-1}(\hat M,\mathbf Z)=0$, which contains a geodesic line $\hat\gamma:\mathbf R\to \hat M$, i.e. $\dist(\hat\gamma(s),\hat\gamma(t))=|s-t|$. Let $\hat \rho_i$, $i=1,2$, be two smooth functions on $\hat M$ such that
$$
|\hat\rho_1(\cdot)-\dist(\hat\gamma(0),\cdot)|\leq 1
$$
and
$$
|\hat \rho_2(\cdot)-\dist(\hat\gamma([0,+\infty)),\cdot)|\leq 1.
$$
Then for any regular value $c_2>1$ of $\hat \rho_2$, we have $\hat \rho_1^{-1}(c_1)\cap \hat\rho_2^{-1}(c_2)\neq \emptyset$ for any regular value $c_1$ of $\hat\rho_1$ satisfying $c_1>c_2+2$. In particular, for almost every regular value $c_1>c_2+2$ of $\hat \rho_1$ the intersection
$$\Sigma=\hat\rho_1^{-1}((-\infty,c_1]))\cap \hat\rho_2^{-1}(c_2)$$ is a smooth hypersurface with smooth boundary $S=\hat\rho_1^{-1}(c_1)\cap \hat\rho_2^{-1}(c_2)$.
\end{lemma}
\begin{proof}
For convenience, we write $p_0=\hat\gamma(0)$ and define
$$\hat\gamma_\pm:[0,+\infty)\to \hat M,\quad t\mapsto\hat\gamma(\pm t).$$
Let us denote the component of $\hat\rho_1^{-1}((-\infty,c_1])$ containing the point $p_0$ by $\hat V$. Note that the distance between $\hat \rho_1^{-1}(c_1)$ and $p_0$ cannot exceed $c_1+1$. So there are positive constants $t_\pm$ such that $\hat\gamma_{+}(t_+)$ and $\hat\gamma_-(t_-)$ are in $\partial\hat V$ while $\hat\gamma_+((t_+,+\infty))$ and $\hat\gamma_-((t_-,+\infty))$ are outside $\hat V$.
\begin{figure}[htbp]
\centering
\includegraphics[width=7cm]{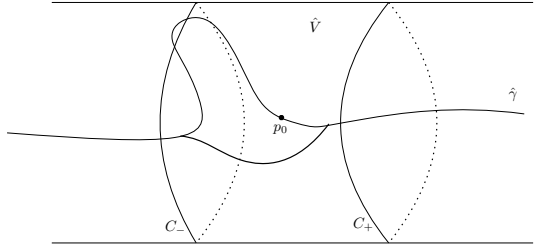}
\caption{The modification of $\hat\gamma$}
\label{Fig: 2}
\end{figure}

First we claim that $\hat\gamma_\pm(t_\pm)$ are contained in the same component of $\partial \hat V$.
Otherwise the points $\hat\gamma_\pm(t_\pm)$ are contained in different components, denoted by $C_\pm$ respectively. As shown in Figure \ref{Fig: 2}, we can modify the geodesic line $\hat\gamma$ to have only one intersection with $C_+$. In particular, this yields $H_{n-1}(\hat M,\mathbf Z)\neq 0$ and contradicts to our assumption.

Now we take a path $P$ in $\partial\hat V$ connecting $\hat\gamma_+(t_+)$ and $\hat\gamma_-(t_-)$. To conclude $\hat \rho_1^{-1}(c_1)\cap \hat\rho_2^{-1}(c_2)\neq \emptyset$, all we need to show are
$$
\dist(\hat\gamma_+(t_+),\hat\gamma_+)<c_2-1\quad \text{and}\quad \dist(\hat\gamma_-(t_-),\hat\gamma_+)>c_2+1.
$$
These inequalities are obvious since $\dist(\hat\gamma_+(t_+),\hat\gamma_+)=0<c_2-1$ and
$$
\dist(\hat\gamma_-(t_-),\hat\gamma_+)\geq t_-\geq c_1-1>c_2+1.
$$
For almost every regular value $c_1$ of $\hat \rho_1$, it is also a regular value the limited map of $\hat \rho_1$ on $\hat\rho_2^{-1}(c_2)$. Therefore, the hypersurface
$\Sigma=\hat\rho_1^{-1}((-\infty,c_1]))\cap \hat\rho_2^{-1}(c_2)$ has a smooth boundary $S=\hat\rho_1^{-1}(c_1)\cap \hat\rho_2^{-1}(c_2)$.
\end{proof}

Next we show that the boundary $S$ above has a "hole" with definite size in $\hat M$ and this needs the notion of (relative) filling radius.
\begin{definition}
Let $C$ be a $k$-cycle in a metric space $(X,d)$. The (relative) filling radius of $C$ in $(X,d)$ is defined by
$$
\fillrad(C;X)=\inf\left\{r>0\,|\,i_*:H_k(C,\mathbf Z)\to H_k(B_r(C),\mathbf Z)\mbox{ is zero}\right\}.
$$
\end{definition}
\begin{lemma}[Filling radius estimate]\label{Lem: filling radius}
Let $(\hat M,\hat g)$ and $S$ be the same as in Lemma \ref{Lem: extrinsic cutting}. Then for any integer $k\neq 0$ the filling radius of the chain $kS$ in $(\hat M,\hat g)$ satisfies
$$
\fillrad(kS;\hat M)\geq \min\{c_1-2c_2-3,c_2-1\}.
$$
\end{lemma}
\begin{proof}
First we show $\dist(S,\hat\gamma)\geq \min\{c_1-2c_2-3,c_2-1\}$. Since we have
$$\dist(S,\hat \gamma_+)\geq c_2-1\geq \min\{c_1-2c_2-3,c_2-1\},$$ it suffices to deal with the case when there is a point $p_1$ on $\hat\gamma_-$ such that
$$
\dist(S,p_1)=\dist(S,\hat\gamma).
$$
Let $p_2$ be a point on $\hat\gamma_+$ such that $\dist(S,p_2)\leq c_2+1$.
As shown in Figure \ref{Fig: 3}, it follows from triangle inequality and minimizing property of $\hat\gamma$ that
\[
\begin{split}
c_1-1\leq\dist(S,p_0)\leq \dist(S,p_1)+2\dist(S,p_2)\leq \dist(S,p_1)+2c_2+2.
\end{split}\]
This yields the desired estimate.
\begin{figure}[htbp]
\centering
\includegraphics[width=10cm]{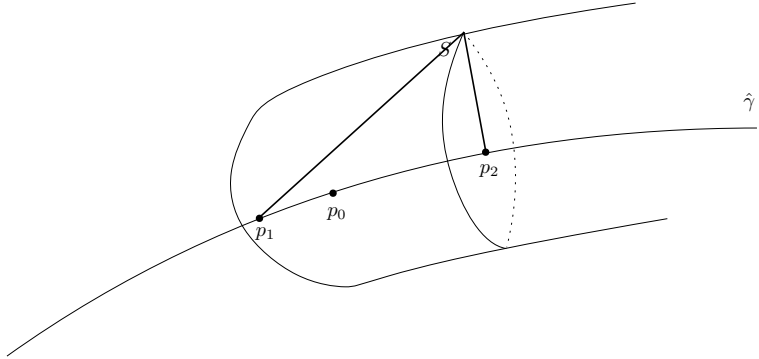}
\caption{The nearest points to $S$}
\label{Fig: 3}
\end{figure}

Now it is easy to deduce
$$
\fillrad(kS;\hat M)\geq \min\{c_1-2c_2-3,c_2-1\}.
$$
Otherwise there is a positive constant $r_0<\dist(S,\hat\gamma)$ such that the map
$$i_*:H_{n-2}(S,\mathbf Z)\to H_{n-2}(B_{r_0}(S),\mathbf Z)$$
is zero.
So we can find a chain $\hat C$ with support in $B_{r_0}(S)$ such that $\partial \hat C=kS$. Let $\Sigma$ be the same as in Lemma \ref{Lem: extrinsic cutting}. This implies that $k\Sigma-\hat C$ is a $(n-1)$-cycle having intersection number $k$ with $\hat\gamma$. In particular, $k\Sigma-\hat C$ is homologically non-trivial and we obtain $H_{n-1}(\hat M,\mathbf Z)\neq 0$, which contradicts to the assumption on $\hat M$.
\end{proof}

\begin{definition}
Let $(\hat M^n,\hat g)$ be a complete Riemannian manifold. It is said to have $C^0$-bounded geometry if there are positive constants $\delta$ and $C$ such that for any point $p$ in $\hat M$ there is a diffeomorphism $\Phi:B_\delta(p)\to B_\delta\subset \mathbf R^n$ such that $C^{-1}\Phi^*g_{euc}\leq \hat g\leq C\Phi^*g_{euc}$.
\end{definition}

\begin{lemma}[Pasting lemma]\label{Lem: pasting}
Let $n\leq 7$ be an integer. If $(\hat M^n,\hat g)$ is a smooth orientable Riemannian manifold with $C^0$-bounded geometry and $S$ is an oriented smooth $(n-2)$-submanifold in $\hat M$ such that $S=\partial C$ for some $(n-1)$-cycle $C$, then there is an oriented smooth embedded area-minimizing hypersurface $\Sigma_{min}$ with $\partial \Sigma_{min}=S$.
\end{lemma}
\begin{proof}
The proof is based on the geometric measure theory and the audience can refer to \cite{S1983} for basic concepts. First let us deal with the case when $\hat M$ is compact. For our purpose we consider the class
$$
\mathcal I_S:=\{\mbox{rectifiable current $T$ with integer multiplicity and $\partial T=S$}\}.
$$
Take a sequence $T_j$ in $\mathcal I_S$ such that
$$\mathbb M(T_j)\to \inf_{T\in \mathcal I_S}\mathbb M(T)\mbox{ as } j\to+\infty.$$
Now it follows from Federer-Fleming compactness theorem (refer to \cite[Theorem 27.3]{S1983}) that $T_j$ converges to $T_\infty$ subsequently in the sense of current. As a consequence, $\partial T_\infty =S$ and the lower-semicontinuity of mass yields
$$
\mathbb M(T_\infty)= \inf_{T\in \mathcal I_S}\mathbb M(T).
$$
In particular, $T_\infty$ is an area-minimizing current and the regularity theorem (see \cite[Theorem 37.7]{S1983}) yields that the support of $T_\infty$ is an oriented smooth embedded minimal hypersurface when $n\leq 7$, denoted by $\Sigma_{\infty}$. From the constancy theorem (see \cite[Theorem 26.27]{S1983}) we conclude that $T_\infty$ has multiplicity $\pm 1$. If we take $\Sigma_{min}=\Sigma_\infty$, then it satisfies all our requirements.

Next we deal with the case when $\hat M$ is non-compact and the key thing is to avoid the sequence $T_j$ drifting to infinity. For this purpose we take a compact exhaustion $\{U_k\}_{k=1}^\infty$ of $\hat M$ and deform the metric $\hat g$ a little bit around $\partial U_k$ to a new metric $\hat g_k$ such that the boundary $\partial U_k$ is mean convex with respect to the outward unit normal in $(U_k,\hat g_k)$. Moreover, we can require
\begin{equation}\label{Eq: equivalence}
\frac{1}{4}\hat g\leq \hat g_k\leq 4\hat g\mbox{ as quadratic forms in } U_k.
\end{equation}
Since the boundary $\partial U_k$ serves as a barrier, we can find an oriented smooth embedded area-minimizing hypersurface $\Sigma_k$ with $\partial \Sigma_k=S$ with respect to the metric $\hat g_k$. Since the hypersurface $\Sigma_k$ is area-minimizing and $(\hat M,\hat g)$ has $C^0$-bounded geometry, we claim that there is a positive constant $\tau$ such that
$$\mathbb M_{\hat g_k}(\Sigma_k\cap B_\delta(p))\geq \tau$$
for any $p\in\Sigma_k$ with $B_\delta(p)\cap S=\emptyset$, where $B_\delta(p)$ is the geodesic $\delta$-ball centered at $p$ with respect to the metric $\hat g$. From the definition of $C^0$-bounded geometry there is a diffeomorphism $\Phi:B_\delta(p)\to B_\delta\subset \mathbf R^n$ such that $C^{-1}\Phi^*g_{euc}\leq \hat g\leq C\Phi^*g_{euc}$ for some universal positive constant $C$. It follows from \cite[Theorem 30.1]{S1983} that for any $0<s<\delta$ we can find a rectifiable current $T$ with integer multiplicity such that $\partial T=\Sigma_k\cap \partial B_s(p)$ and
$$
\mathbb M_{g_{euc}}(\Sigma_k\cap \partial B_s(p))\geq c(n)\mathbb M_{g_{euc}}(T)^{\frac{n-2}{n-1}}
$$
for a universal constant $c(n)$ depending only on $n$. From \eqref{Eq: equivalence} the area-minimizing property of $\Sigma_k$ with respect to $\hat g_k$ we see
\[
\begin{split}
\frac{\mathrm d}{\mathrm ds}\mathbb M_{\hat g_k}(\Sigma_k\cap B_s(p))&\geq \mathbb M_{\hat g_k}(\Sigma_k\cap \partial B_s(p))\\
&\geq (2C)^{2-n}\mathbb M_{g_{euc}}(\Sigma_k\cap \partial B_s(p))\\
&\geq (2C)^{2-n}c(n) \mathbb M_{g_{euc}}(T)^{\frac{n-2}{n-1}}\\
&\geq (2C)^{3-2n}c(n)\mathbb M_{\hat g_k}(T)^{\frac{n-2}{n-1}}\\
&\geq (2C)^{3-2n}c(n)\mathbb M_{\hat g_k}(\Sigma_k\cap B_s(p))^{\frac{n-2}{n-1}}.
\end{split}
\]
As a consequence, we have
$$
\frac{\mathrm d}{\mathrm ds}\left(\mathbb M_{\hat g_k}(\Sigma_k\cap B_s(p))^{\frac{1}{n-1}}\right)\geq \frac{1}{n-1}(2C)^{3-2n}c(n)
$$
and so
$$
\mathbb M_{\hat g_k}(\Sigma_k\cap B_\delta(p))^{\frac{1}{n-1}}\geq \tau:=\left(\frac{1}{n-1}(2C)^{3-2n}c(n)\delta\right)^{n-1}.
$$
Take a fixed rectifiable current $C$ with integer multiplicity and $\partial C=S$. Again from \eqref{Eq: equivalence} and the area-minimizing property of $\Sigma_k$ with respect to $\hat g_k$ we have
$$\mathbb M_{\hat g_k}(\Sigma_k)\leq 2^{n-1}\mathbb M_{\hat g}(C).$$
This implies that the diameter of $\Sigma_k$ with respect to $\hat g$ is uniformly bounded from above. Otherwise we can find sufficiently many disjoint $\delta$-balls where $\Sigma_k$ has mass no less than $\tau$ in each ball and finally this leads to a contradiction to the mass bound above. Since $\hat g_k$ is only different from $\hat g$ around boundary $\partial U_k$, the hypersurface $\Sigma_k$ is area-minimizing with respect to $\hat g$ in a slightly smaller region than $U_k$. From the uniform diameter estimate hypersurfaces $\Sigma_k$ are contained in a fixed compact subset of $\hat M$. So we can take $\Sigma_{min}$ to be the limit of $\Sigma_k$ up to a subsequence and it satisfies all our desired requirements.
\end{proof}
\subsection{$T^l$-stablized scalar curvature lower bound and intrinsic cutting}
We start with the following definition coming from \cite{Gromov2020}.
\begin{definition}
Let $(M,g)$ be a Riemannian manifold and $\sigma:M\to \mathbf R$ a continuous function on $M$. We say that $(M,g)$ has {\it $T^l$-stablized scalar curvature lower bound $\sigma$} if there are $l$ positive smooth functions $u_1,\ldots,u_l$ such that the warped metric
$$
g_{warp}=g+\sum_{i=1}^lu_i^2\mathrm d\theta_i^2
$$
on $M\times T^l$ satisfies $R(g_{warp})\geq \sigma$.
\end{definition}

\begin{lemma}[Intrinsic cutting lemma]\label{Lem: intrinsic cutting}
Let $(\Sigma^k,g_\Sigma)$, $k\leq 7$, be a compact Riemannian manifold with $T^l$-stablized scalar curvature lower bound $\sigma$ and $S=\partial\Sigma\neq \emptyset$. Then for any
$
r_0>0
$
we have the following alternative:
\begin{itemize}
\item either we have $\fillrad(S;\Sigma)\leq r_0$;
\item or we can find an embedded hypersurface $S_{r_0}$ homologous to $S$ such that
\begin{itemize}
\item $S_{r_0}$ is contained in $B_{r_0}(S)$;
\item $S_{r_0}$ with the induced metric has $T^{l+1}$-stablized scalar curvature lower bound
\begin{equation}\label{Eq: warped trick}
\sigma|_{S_{r_0}}-\frac{4(n-1)\pi^2}{nr_0^2},\mbox{ where } n=k+l.
\end{equation}
\end{itemize}
\end{itemize}
\end{lemma}
\begin{proof}
Let us construct the desired hypersurface $S_{r_0}$ when $\fillrad(S;\Sigma)> r_0$. In this case, the region $M:=B_{r_0}(S)$ has two parts of boundary $\partial_-=S$ and $\partial_+=\partial M-S$. Clearly we have $\dist(\partial_-,\partial_+)=r_0$ and so we can construct a smooth function $\rho:M\to (-r_0/2,r_0/2)$ with $\Lip \rho\leq 1$ and $\rho^{-1}(\pm r_0/2)=\partial_\pm$.

Let us take
$$
h:\left(-\frac{r_0}{2},\frac{r_0}{2}\right)\to \mathbf R,\quad t\mapsto -\frac{2(n-1)\pi}{nr_0}\tan\left(\frac{\pi t}{r_0}\right) ,\quad n=k+l.
$$
In particular, the function $h$ satisfies
\begin{equation}\label{Eq: h equation}
\frac{n}{n-1}h^2-2|\mathrm d h|=-\frac{4(n-1)\pi^2}{nr_0^2}.
\end{equation}
Denote $\Omega_0=\{\rho<0\}$. We consider the minimizing problem for functional
\begin{equation*}
\mathcal A^h(\Omega)=\mathcal H^{n-1}_{g_{warp}}(\partial^*\Omega\times T^l)-\int_{M\times T^l}(\chi_\Omega-\chi_{\Omega_0})h\circ \rho\,\mathrm d\mathcal H^n_{g_{warp}}
\end{equation*}
among the following collection of Caccioppoli sets
\begin{equation*}
\mathcal C=\{\text{Caccioppoli sets $\Omega$ in $M$ such that $\Omega\Delta \Omega_0\Subset M-\partial M$}\}.
\end{equation*}
Here functions on $M$ are viewed as $T^l$-invariant functions on $M\times T^k$ in the definition of functional $\mathcal A^h$. It follows from \cite[Proposition 2.1]{Zhu2021} that there is a Caccioppoli set $\Omega_{min}$ minimizing the functional $\mathcal A^h$ in the class $\mathcal C$. Now we explain the reason why $\Omega_{min}$ has smooth boundary when $k\leq 7$ by viewing it as a minimizer of another functional on $M$. Recall that the warped metric $g_{warp}$ has the form
$$
g_{warp}=g_\Sigma+\sum_{i=1}^lu_i^2\mathrm d\theta_i^2.
$$
In particular, we can write the functional $\mathcal A^h$ as
$$
\mathcal A^h(\Omega)=\int_{\partial^*\Omega}\left(\prod_{i=1}^lu_i\right)\mathrm d\mathcal H^{k-1}_g-\int_M(\chi_\Omega-\chi_{\Omega_0})\left(\prod_{i=1}^lu_i\right)h\circ \rho\,\mathrm d\mathcal H^{k}_g.
$$
Let
$$
g_{conf}=\left(\prod_{i=1}^lu_i\right)^{\frac{2}{k-1}}g\mbox{ and } v=\left(\prod_{i=1}^lu_i\right)^{-\frac{1}{k-1}}.
$$
Then the set $\Omega_{min}$ can be viewed as a minimizer of the functional
$$
\mathcal H^{k-1}_{g_{conf}}(\partial^*\Omega)-\int_{M}(\chi_{\Omega}-\chi_{\Omega_0})vh\circ\rho\,\mathrm d\mathcal H^k_{g_{conf}}.
$$
As a consequence the smoothness of $\Omega_{min}$ follows from the geometry measure theory (see \cite[Theorem 4.2]{Duzaar1993}) when $k\leq 7$.

Denote $S_{r_0}=\partial\Omega_{min}$. Clearly $S_{r_0}$ Now it remains to show \eqref{Eq: warped trick} for  Denote $\nu$ to be the outward unit normal of $\partial\Omega_{min}$ with respect to $\Omega_{min}$ and let $X$ be a smooth vector field on $\Sigma$ with compact support whose limit on $\partial\Omega_{min}$ equals to $\psi\nu$ for some smooth function $\psi$. Let $\Phi:(-\epsilon,\epsilon)\times\Sigma \to \Sigma$ be the flow generated by $X$ and $\Omega_t=\Phi(t,\Omega_{min})$. It is not difficult to compute the first variation formula
$$
\left.\frac{\mathrm d}{\mathrm dt}\right|_{t=0}\mathcal A^h(\Omega_t)=\int_{S_{r_0}\times T^l}(H-h\circ \rho)\psi \,\mathrm d\sigma=0,
$$
where $H$ is the mean curvature of $S_{r_0}\times T^l$ with respect to the outward unit normal vector field.
The arbitrary choice of the vector field $X$ (consequently as well as the function $\psi$) implies that $H=h\circ \rho$. Now we can further calculate
\[
\begin{split}
\left.\frac{\mathrm d^2}{\mathrm dt^2}\right|_{t=0}\mathcal A^h(\Omega_t)=&\int_{S_{r_0}\times T^l}|\nabla\psi|^2\\
&-\frac{1}{2}\left(R(g_{warp})-R+H^2+|A|^2-2\nu(h\circ\rho)\right)\psi^2\,\mathrm d\sigma\geq 0,
\end{split}
\]
where $R$ is the scalar curvature of $S_{r_0}\times T^l$ with respect to the induced metric and $A$ is the second fundamental form of $S_{r_0}\times T^l$ with respect to the outward unit normal vector field. Define
$$
\mathcal L=-\Delta-\frac{1}{2}\left(R(g_{warp})-R+H^2+|A|^2-2\nu(h\circ\rho)\right).
$$
Again from the arbitrary choice of the vector field $X$ we see that the operator $\mathcal L$ is nonnegative. Take $u_{l+1}$ to be the first eigenfunction with respect to $\mathcal L$ whose corresponding first eigenvalue is denoted by $\lambda_1$. It is clear that $u_{l+1}$ is a positive smooth function. It follows from \cite[Proposition 11.14]{GL1983} that
\begin{equation*}
\begin{split}
&R\left(g_\Sigma|_{S_{r_0}}+\sum_{i=1}^l(u_i|_{S_{r_0}})^2\mathrm d\theta_i^2+u_{l+1}^2\mathrm d\theta_{l+1}^2\right)\\
=&R-\frac{2\Delta u_{l+1}}{u_{l+1}}\\
=&R(g_{warp})+H^2+|A|^2-2\nu(h\circ\rho)+2\lambda_1\\
\geq & \sigma|_{S_{r_0}}+\left(\frac{n}{n-1}h^2-2|\mathrm d h|\right)\circ \rho\\
=&\sigma|_{S_{r_0}}-\frac{4(n-1)\pi^2}{nr_0^2},
\end{split}
\end{equation*}
where we use the facts $|\nu(h\circ \rho)|\leq |\mathrm dh|\circ \rho$ and
$$
|A|^2\geq \frac{1}{n-1}H^2,\mbox{ where }n=k+l.
$$
This completes the proof.
\end{proof}

As a preparation we also need to mention Gromov's width estimate from \cite{Gromov2018}. Recall that a compact manifold $(M,\partial_\pm)$ with $\partial M=\partial_-\sqcup\partial_+$ is called an over-torical band if there is a non-zero degree map
$$
f:(M,\partial_\pm)\to\left(T^{n-1}\times[-1,1],T^{n-1}\times \{\pm 1\}\right).
$$
\begin{proposition}
If $(M^n,g,\partial_\pm)$ is an over-torical band with scalar curvature $R(g)\geq \sigma_0>0$, then the width
$$
\width(M,g):=\dist(\partial_-,\partial_+)\leq 2\pi\sqrt{\frac{n-1}{n\sigma_0}}<\frac{2\pi}{\sqrt{\sigma_0}}.
$$
\end{proposition}
This width estimate has the following immediate corollary.
\begin{corollary}\label{Cor: diameter estimate}
If $(S^2,g_S)$ is a closed surface with $T^l$-stablized scalar curvature lower bound $\sigma \geq \sigma_0>0$, then its diameter cannot exceed $2\pi/\sqrt{\sigma_0}$. If $(S,g_S)$ is a compact surface with boundary, then above diameter estimate holds if $\partial S\times T^l$ is mean convex with respect to the $T^l$-warped metric.
\end{corollary}
\begin{proof}
Let $p$ and $q$ be any pair of points in $S$. For our purpose we minimize the area functional among hypersurfaces in $(S\times T^l,g_{warp})$ in the form of $\gamma\times T^l$ for a curve $\gamma\subset S$ whose boundary is exactly $\{p,q\}\times T^l$. From geometric measure theory we can find a smooth area-minimizing hypersurface $\gamma_0\times T^l$ in $(S\times T^l,g_{warp})$, where $\gamma_0$ is a curve connecting $p$ and $q$. From a similar calculation as in the proof of Lemma \ref{Lem: intrinsic cutting} we know that $\gamma_0$ has $T^{l+1}$ stablized scalar curvature lower bound $\sigma\geq \sigma_0$. From Gromov's width estimate we have
$$
\length(\gamma_0)=\width(\gamma_0\times T^{l+1})<\frac{2\pi}{\sqrt{\sigma_0}}.
$$
This completes the proof.
\end{proof}

\subsection{Slice-and-dice argument}
In this subsection, let us make a further discussion on the hypersurface $S_{r_0}$ from intrinsic cutting in the Riemannian manifold $(M,g)$. First we recall the slice-and-dice result from \cite{CL2020}.
\begin{proposition}\label{Prop: slice and dice}
Let $(S,g_S)$ be a closed Riemannian $3$-manifold with $T^l$-stablized scalar curvature lower bound $\sigma\geq \sigma_0>0$. Then for any $0<\epsilon<\sigma_0$ there are finitely many embedded $2$-spheres $S_1$, ..., $S_p$ and embedded $2$-disks $D_1$, ..., $D_q$ such that
\begin{itemize}
\item the diameters of all spheres $S_i$ and disks $D_j$ with respect to induced metric are bounded from above by
$
2\pi/\sqrt{\sigma_0-\epsilon}
$;
\item we have
$$
\int_{S_i}(\sigma|_{S_i}-\epsilon)\,\mathrm d\mu\leq 8\pi\mbox{ and } \int_{D_j}(\sigma|_{D_j}-\epsilon)\,\mathrm d\mu\leq 4\pi;
$$
\item spheres $S_i$ are pairwise disjoint, disks $D_j$ are also pairwise disjoint but intersect one of $S_i$ transversely on its boundary $\partial D_j$;
\item the diameter of each component of the complement
$$
S-\left(\bigcup_{i=1}^p S_i\right)\cup\left(\bigcup_{j=1}^q D_j\right)
$$
is no greater than
$$
\frac{2\pi}{\sqrt{\sigma_0-\epsilon}}+\frac{8\pi}{\sqrt\epsilon},
$$
where the diameter is computed with respect to the distance in $S$.
\end{itemize}
\end{proposition}
\begin{proof}
Let us go through the proof from \cite{CL2020} for completeness. The proof will be divided into the following two steps.

{\it Step 1. Topology reduction.} We claim that there are pairwise disjoint embedded $2$-spheres $S_1$, ..., $S_{k}$ with diameter no greater than $2\pi/\sqrt{\sigma_0}$ and
$$
\int_{S_i}\sigma|_{S_i}\,\mathrm d\mu \leq 8\pi
$$
such that the inclusion $i:H_2(\partial\hat S,\mathbf Z)\to H_2(\hat S,\mathbf Z)$ is surjective, where $\hat S$ is denoted to be the metric completion of $S-\bigcup_i S_i$. By definition, there are smooth positive functions $u_1$, ..., $u_l$ such that the warped metric
$$
g_{warp}=g_S+\sum_{i=1}^lu_i^2\mathrm d\theta_i^2
$$
has scalar curvature $R(g_{warp})\geq \sigma\geq \sigma_0>0$. If $S$ satisfies $H_2(S,\mathbf Z)=0$, then we are already done. So we just need to deal with the case when $H_2(S,\mathbf Z)\neq0$. After fixing a non-zero class $\beta\in H_2(S,\mathbf Z)$ we can minimize the area functional among all smooth hypersurfaces in manifold $(S\times T^l,g_{warp})$ having the form of $\Sigma\times T^l$ with $[\Sigma]=\beta$. Again the existence of a smooth minimizer is guaranteed by the geometric measure theory. Take one component of this minimizer and we denote it by $S_1\times T^l$. Through a similar argument as in the proof of Lemma \ref{Lem: intrinsic cutting} it is easy to deduce that $(S_1,g_{S_1})$ has $T^{l+1}$-stablized scalar curvature lower bound
$\sigma|_{S_1}$. It follows from the computation as in \cite[Lemma 2.3]{Zhu2020} that
$$
\int_{S_1}\sigma|_{S_1}\,\mathrm d\mu\leq 4\pi\chi(S_1)\leq 8\pi.
$$
In particular, the area of $S_1$ is bounded from above by $8\pi\sigma_{0}^{-1}$. We also know from Corollary \ref{Cor: diameter estimate} that the diameter of $S_1$ is no greater than $2\pi/\sqrt{\sigma_0}$.

From now on there are two possibilities:
either the metric completion $\hat S$ of $S-S_1$ satisfies the desired property and we are done, or the inclusion map $i:H_2(\partial\hat S,\mathbf Z)\to H_2(\hat S,\mathbf Z)$ is not surjective and we can find another class $\beta$ not contained in the image of $i$. The procedure above can be conducted inductively since $\hat S$ is a Riemannian manifold with minimal boundary, where the boundary serves as a barrier for area minimizing problem above. We just need to show that the construction above must terminate after repeated for finite times\footnote{Here we just apply the compactness argument from \cite{CL2020}, which was the only argument available when this paper was written. Later, Bamler, Li and Mantoulidis were able to simplify the proof with a careful analysis on topology, see \cite[Lemma 2.5]{BLM2022}.}. If this is not true, there would be a sequence of $2$-spheres $S_i$ from the construction above such that for any integer $i_0>0$ the class $[S_{i_0}]$ cannot expressed as a linear combination of $\{[S_i]\}_{i=1}^{i_0-1}$ in $H_2(S,\mathbf Z)$. Since all $S_i$ have uniformly bounded area (no greater than $8\pi\sigma_0^{-1}$) and they are stable embedded minimal surface with respect to the conformal metric $\left(\prod_{i=1}^lu_i\right)g_S$, their second fundamental forms are uniformly bounded (refer to \cite{SSY1975} and \cite{SS1981}). Up to a subsequence $S_i$ converges to a limit surface $S_\infty$ as one-sheeted graphs. So there are infinitely many $S_i$ representing the same homology class in $H_2(S,\mathbf Z)$ and this leads to a contradiction.

{\it Step 2. Cutting $\hat S$ into small pieces.} In general, a compact $3$-manifold with stablized scalar curvature lower bound $\sigma_0>0$ may not have diameter bounded from above. So the manifold $\hat S$ obtained above can be quite long and the strategy is to cut it into smaller pieces.

Let us start with a fixed boundary component $\partial_0$ of $\partial\hat S$. In the case when $\hat S$ has no boundary, we deal with $\hat S-B_\delta$ instead of $\hat S$ for a fixed small geodesic ball $B_\delta$. In the following, we divide the discussion into two cases. If $\hat S$ is contained in the $(4\pi/\sqrt{\epsilon})$-neighborhood of $\partial_0$, then it follows from the diameter estimate of $\partial_0$ that the diameter of $\hat S$ is no greater than
$$\frac{2\pi}{\sqrt{\sigma_0}}+\frac{8\pi}{\sqrt\epsilon}.$$
Otherwise the complement $\hat S-B(\partial _0;4\pi/\sqrt{\epsilon})$ is non-empty and we would like to find suitable surface for cutting. Take
$$V=B(\partial _0;4\pi/\sqrt{\epsilon})-B(\partial _0;2\pi/\sqrt{\epsilon}).$$
Let
$$
h:V\to (-\infty,+\infty),\quad x\mapsto \frac{l+2}{l+3}\sqrt \epsilon\tan\left(\frac{\sqrt\epsilon}{2}(\dist(\cdot,\partial_0)-3\pi/\sqrt\epsilon)\right).
$$
Denote $\Omega_0=B(\partial _0;3\pi/\sqrt{\epsilon})$. We can minimize the functional
$$
\mathcal A^h(\Omega)=\mathcal H^{l+2}_{g_{warp}}(\partial^*\Omega\times T^l)-\int_{\hat S\times T^l}(\chi_\Omega-\chi_{\Omega_0})h\,\mathrm d\mathcal H^{l+3}_{g_{warp}}
$$
among the following collection of Caccioppoli sets
\begin{equation*}
\mathcal C=\{\text{Caccioppoli sets $\Omega$ in $\hat S$ such that $\Omega\Delta \Omega_0\Subset V$}\}.
\end{equation*}
From geometric measure theory we can find a smooth minimizer $\Omega_{min}$ whose boundary $\partial\Omega_{min}$ is an embedded surface possibly with free boundary. Also $\partial\Omega_{min}$ is homologous to $\partial_0$ in $H_2(\hat S,\partial\hat S,\mathbf Z)$. Similar as in Lemma \ref{Lem: intrinsic cutting}, $\partial\Omega_{min}$ has $T^{l+1}$ stablized scalar curvature lower bound
$$
\sigma|_{\partial\Omega_{min}}-\frac{l+2}{l+3}\epsilon\geq \sigma_0-\epsilon>0.
$$
Again it follows from the computation as in \cite[Lemma 2.3]{Zhu2020} that for each component $\Sigma_{min}$ of $\partial\Omega_{min}$ it holds
$$
\int_{\Sigma_{min}}(\sigma|_{\Sigma_{min}}-\epsilon)\leq 4\pi\chi(\Sigma_{min}).
$$
Moreover, we conclude from Corollary \ref{Cor: diameter estimate} that each component of $\partial\Omega_{min}$ has diameter no greater than $2\pi/\sqrt{\sigma_0-\epsilon}$.

Since the region $\Omega_{min}$ may have multiple components, we take $\Omega_1$ to be the component containing $\partial_0$. In the following, we denote $\partial_{int}\Omega=\partial\Omega-\partial\hat S$ for any region $\Omega\subset \hat S$. From the construction we see that $B(\partial_0,2\pi/\sqrt{\epsilon})\subset\Omega_1$. With the help of the topology reduction we claim that each component of $\hat S-\Omega_1$ contains only one component of $\partial_{int}\Omega_1$ in its boundary. Otherwise we can find a simple closed curve in $\hat S$ such that it intersects some component of $\partial_{int}\Omega_1$ only once. This just yields $H_2(\hat S,\partial\hat S,\mathbf Z)\neq 0$. However we have the exact sequence
$$
H_2(\partial\hat S,\mathbf Z)\to H_2(\hat S,\mathbf Z)\to H_2(\hat S,\partial\hat S,\mathbf Z)\to H_1(\partial\hat S,\mathbf Z),
$$
where the surjectivity of the first map along with the fact $H_1(\partial\hat S,\mathbf Z)=0$ yields $H_2(\hat S,\partial\hat S,\mathbf Z)=0$. This leads to a contradiction.

Now the procedure above can be repeated for each component of $\hat S-\Omega_1$ and inductively we end up with an exhaustion
$$
\Omega_1\subset \Omega_2\subset \cdots \subset\Omega_m\subset\hat S,
$$
where the finiteness of this exhaustion comes from the facts $B(\partial_0,2\pi i/\sqrt\epsilon)\subset \Omega_i$ and that $\hat S$ has bounded diameter. The proof is now completed by taking $S_i$ to be those $2$-spheres coming from the topology reduction and $\partial_{int}\Omega_i$, and taking $D_j$ to be those disks from $\partial_{int}\Omega_i$ (see Figure \ref{Fig: 1}).
\end{proof}
\begin{figure}[htbp]
\centering
\includegraphics[width=9cm]{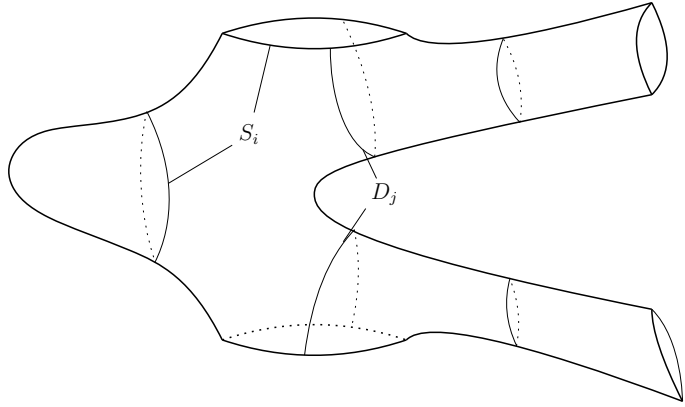}
\caption{The slice-and-dice procedure}
\label{Fig: 1}
\end{figure}



\section{Proof of main theorems}
\subsection{Selection argument}
\begin{definition}
We say that a Riemannian manifold $(M,g)$ satisfies the uniform filling property if there is a function $F:(0,+\infty) \to (0,+\infty)$ such that if $C$ is a cycle contained in some geodesic ball $B_r(p)$ homologous to zero in $H_k(M,\mathbf Z)$, then it is also homologous to zero in $H_k(B_{F(r)}(p),\mathbf Z)$.
\end{definition}
\begin{proposition}\label{Prop: selection}
Let $(\hat M^5,\hat g)$ be a complete Riemannian manifold with $C^0$-bounded geometry, uniform filling property and $H_k(\hat M,\mathbf Z)=0$ for all $k\geq 3$. Assume that $(\hat M,\hat g)$ contains a geodesic line $\hat\gamma$ and has $T^l$-stablized scalar curvature lower bound $\sigma\geq \sigma_0>0$ for some constant $\sigma_0$. Then for any $\tilde\epsilon>0$ we can find a homotopically non-trivial piecewise smooth $2$-sphere $S_\epsilon$ in $\hat M$ such that
$$
\int_{S_{\tilde\epsilon}}\sigma|_{S_{\tilde \epsilon}}\,\mathrm d\mu\leq 8\pi+\tilde\epsilon.
$$
In particular we have $\pi_2(\hat M)\neq 0$.
\end{proposition}
\begin{proof}
From differential topology it is standard to construct smooth functions $\hat\rho_1$ and $\hat\rho_2$ such that
$$
|\hat\rho_1(\cdot)-\dist(\hat\gamma(0),\cdot)|\leq 1
$$
and
$$
|\hat \rho_2(\cdot)-\dist(\hat\gamma([0,+\infty)),\cdot)|\leq 1.
$$
The audience can refer to \cite[Theorem 2.2 on page 44]{Hirsch1976} for details.
Let $L$ be a large positive constant to de determined later. Then we take $c_2$ to be a regular value of $\hat\rho_2$ contained in $(L+1,L+2)$ and $c_1$ to be a regular value of $\hat\rho_1$ contained in $(3L+7,3L+8)$. It follows from Lemma \ref{Lem: extrinsic cutting} and Lemma \ref{Lem: filling radius} that we can find an embedded closed 3-submanifold $S\subset \hat M$ homologous to zero in $H_3(\hat M,\mathbf Z)$ such that
$$\fillrad(S,\hat M)\geq \min\{c_1-2c_2-3,c_2-1\}\geq L.$$
From Lemma \ref{Lem: pasting} there is an embedded area-minimizing hypersurface $\Sigma$ with $\partial \Sigma=S$. In particular, $\Sigma$ has $T^{l+1}$-stablized scalar curvature lower bound $\sigma|_\Sigma\geq\sigma_0>0$.

Let $r_0$ be another positive constant to be determined later.
It follows from Lemma \ref{Lem: intrinsic cutting} that we can find an embedded hypersurface $S_{r_0}$ homologous to $S$ in $\Sigma$ such that $S_{r_0}$ is contained in $B_{r_0}(S)$ and it has $T^{l+2}$-stablized scalar curvature lower bound
\begin{equation*}
\sigma|_{S_{r_0}}-\frac{4(l+4)\pi^2}{(l+5)r_0^2}.
\end{equation*}
Take
$$
\epsilon=\frac{4\pi^2}{(l+5)r_0^2}.
$$
From Proposition \ref{Prop: slice and dice} we can find finitely many embedded $2$-spheres $S_1$, ..., $S_p$ and embedded $2$-disks $D_1$, ..., $D_q$ in $S_{r_0}$ such that
\begin{itemize}
\item the diameters of all spheres $S_i$ and disks $D_j$ with respect to induced metric are bounded from above by
$$
\Lambda_1:=2\pi\left(\sigma_0-\frac{4\pi^2}{r_0^2}\right)
;$$
\item we have
$$
\int_{S_i}\left(\sigma|_{S_i}-\frac{4\pi^2}{r_0^2}\right)\,\mathrm d\mu\leq 8\pi\mbox{ and } \int_{D_j}\left(\sigma|_{D_j}-\frac{4\pi^2}{r_0^2}\right)\,\mathrm d\mu\leq 4\pi;
$$
\item spheres $S_i$ are pairwise disjoint, disks $D_j$ are also pairwise disjoint but intersect one of $S_i$ transversely on its boundary $\partial D_j$;
\item the diameter of each component of the complement
\begin{equation*}
\hat S:=S_{r_0}-\left(\bigcup_{i=1}^k S_i\right)\cup\left(\bigcup_{j=1}^p D_j\right)
\end{equation*}
is no greater than
$$
\Lambda_2:=2\pi\left(\sigma_0-\frac{4\pi^2}{r_0^2}\right)+4r_0,
$$
where the diameter is computed with respect to the distance in $S_{r_0}$.
\end{itemize}
Let us denote $U_1,\ldots, U_m$ to be the components of the complement $\hat S$. Then the boundary components of $U_k$ can be divided into two classes: an entire $2$-sphere $S_i$ or the union of part of $S_i$ and several disks $D_j$. In each case, the diameter of these boundary components in $S_{r_0}$ cannot exceed $3\Lambda_1$. For later use let us denote $(\partial U_k)_\tau$, $\tau=1,2,\ldots, n_k$, to be all boundary components of $U_k$.

Now we turn to the selection argument for $2$-spheres $S_i$ and disks $D_j$ from the slice-and-dice argument. The discussion will be divided into two cases.

{\it Case 1. There is some $2$-sphere $S_{i}$ homotopically non-trivial in $\hat M$.} Then for this $S_i$ we have
\begin{equation}\label{Eq: GB}
\int_{S_i}\sigma|_{S_i}\,\mathrm d\mu \leq 8\pi\left(1-\frac{4\pi^2}{\sigma_0 r_0^2}\right)^{-1}.
\end{equation}

{\it Case 2. All $2$-spheres $S_i$ are homotopically trivial.} In this case, we divide the discussion further into two subcases.

{\it Case 2a. There is some $2$-sphere $S_i$ separated into two parts $S_i^+$ and $S_i^-$ by some disk $D_j$  such that the $2$-spheres $S_i^+\cup D_j$ and $S_i^-\cup D_j$ are homotopically non-trivial in $\hat M$.} In this case, we point out that $S_i^+\cup D_j$ or $S_i^-\cup D_j$ serves as an candidate for the desired $2$-sphere $S_{\tilde\epsilon}$.
Notice that
$$
\int_{S_i^+\cup D_j}\sigma|_{S_i^+\cup D_j}\,\mathrm d\mu+\int_{S_i^+\cup D_j}\sigma|_{S_i^-\cup D_j}\,\mathrm d\mu\leq 16\pi\left(1-\frac{4\pi^2}{\sigma_0 r_0^2}\right)^{-1}.
$$
As a result at least one of $S_i^+\cup D_j$ and $S_i^-\cup D_j$ satisfies the inequality \eqref{Eq: GB}.

{\it Case 2b. For all $S_i$ separated into two parts $S_i^+$ and $S_i^-$ by some disk $D_j$, one of $S_i^+\cup D_j$ and $S_i^-\cup D_j$ is homotopically trivial.} Since all $S_i$ are homotopically trivial, all $S_i^+\cup D_j$ and $S_i^-\cup D_j$ have to be homotopically trivial at the same time. In particular, each boundary component $(\partial U_k)_\tau$ of $U_k$ is homologous to zero in $\hat M$ since it can be expressed as a sum of some $S_i$ and $S_i^\pm\cup D_j$. From the uniform filling property of $\hat M$ we conclude that $(\partial U_k)_\tau$ can be filled with some chain $V_{k,\tau}$ in $F(3\Lambda_1)$-neighborhood of $(\partial U_k)_\tau$. It is not difficult to see
\begin{itemize}
\item for each $k$ the sum
$$
U_k+\sum_{\tau=1}^{n_k}V_{k,\tau}
$$
is a $3$-cycle contained in a geodesic ball centered at some point in $U_k$ with radius $\Lambda_2+2F(3\Lambda_1)$.
\item if for each $j$ we denote
$$\mathcal C_j=\{V_{k,\tau}:V_{k,\tau} \mbox{ has non-empty intersection with } S_j\} ,$$ then the sum
$$
\sum_{\mathcal C_j}V_{k,\tau}
$$
is a $3$-cycle contained in a geodesic ball centered at some point in $S_j$ with radius $\Lambda_2+2F(3\Lambda_1)$. Here we use the fact $\Lambda_1\leq\Lambda_2$.
\end{itemize}
Combined with the $H_k$-vanishing condition of $\hat M$ for $k\geq 3$ we conclude that that the $3$-cycles $U_k+\sum_{\tau=1}^{n_k} V_{k,\tau}$ and $\sum_{\mathcal C_j}V_{k,\tau}$ can be filled with some chain $W_k$ and $X_j$ respectively in the tubular neighborhood of $S_{r_0}$ with radius $F\left(\Lambda_2+2F(3\Lambda_1)\right)$. As a consequence we have
\[
\begin{split}
S_{r_0}=\sum_{k=1}^m U_k=\sum_{k=1}^m\left(U_k+\sum_{\tau=1}^{n_k}V_{k,\tau}\right)-\sum_{j=1}^p\sum_{\mathcal C_j}V_{k,\tau}=\partial\left(\sum_{k=1}^m W_k-\sum_{j=1}^p X_j\right).
\end{split}
\]
Since $S_{r_0}$ is contained in the $r_0$-neighborhood of $S$, we see that
$$
L\leq \fillrad(S,\hat M)\leq r_0+F\left(\Lambda_2+2F(3\Lambda_1)\right).
$$

Now we complete the proof as follows. For each $\tilde\epsilon >0$, we take $r_0$ to be a large constant such that
$$
\left(1-\frac{4\pi^2}{\sigma_0 r_0^2}\right)^{-1}\leq 1+\tilde\epsilon.
$$
Then we take $L$ to be a large constant such that
$$
L> r_0+F\left(\Lambda_2+2F(3\Lambda_1)\right).
$$
Now we see that only Case 1 and Case 2a can happen and so there is one homotopically non-trivial piecewise smooth $2$-sphere $S_{\tilde \epsilon}$ (given by some $S_i$ or $S_i^\pm\cup D_j$) that satisfies
$$
\int_{S_{\tilde\epsilon}}\sigma|_{S_{\tilde \epsilon}}\,\mathrm d\mu\leq 8\pi+\tilde\epsilon.
$$
This completes the proof.
\end{proof}
\subsection{Proof of Theorem \ref{Thm: main}} We begin with the following lemma.
\begin{lemma}\label{Lem: geometry of cover}
Let $(M,g)$ be a closed Riemannian manifold. Assume that $(\hat M,\hat g)$ is a regular covering space of $(M,g)$. Then $(\hat M,\hat g)$ has $C^0$-bounded geometry and the uniform filling property. If $(\hat M,\hat g)$ is also non-compact, then it contains a geodesic line.
\end{lemma}
\begin{proof}
The $C^0$-bounded geometry comes from Rauch comparison theorem and the fact that $(\hat M,\hat g)$ has positive injective radius and uniformly bounded curvature. The uniform filling property and the existence of geodesic line in the non-compact case are proven in \cite[Lemma 6 and Proposition 10]{CL2020}.
\end{proof}
Next we prove Theorem \ref{Thm: main}.
\begin{proof}[Proof of Theorem \ref{Thm: main}]
The proof will be divided into the following cases.

{\it Case 1. $\Ric(g)\equiv 0$.} It follows from the Cheeger-Gromoll splitting theorem that the universal covering $(\hat M,\hat g)$ splits as $(\hat N^k,\hat h)\times \mathbf R^{n-k}$, where $(\hat N,\hat h)$ is a closed, simply connected, Ricci-flat $k$-manifold. From the $H_i$-vanishing condition of $\hat M$ for $i\geq 3$, we see that $k\leq 2$ and so $\hat M$ is isometric to $\mathbf R^n$ or $\hat N$ is a closed, simply connected, flat surface. Clearly the latter case is impossible since the only simply connected closed surface is $2$-sphere and it does not admit any flat metric.

{\it Case 2. $\Ric(g)\nequiv 0$.} For any smooth metric $h$ on $M$ we consider the operator
$$
\mathcal L_h=-\Delta_h+\frac{R(h)-R(g)}{2}
$$
and denote $\lambda_h$ to be the first eigenvalue with respect to $\mathcal L_h$. Let $\zeta$ be an arbitrary $(0,2)$-tensor on $M$ and $g_t=g-2t\zeta$ for $t\in(-\epsilon,\epsilon)$. From \cite[Theorem 1.174]{Besse2008} we have
$$
\left.\frac{\partial}{\partial t}\right|_{t=0}R(g_t)=2\langle\Ric(g),\zeta\rangle_g-2\Div_g(\Div_g\zeta-\mathrm d\tr_g\zeta)
$$
Let $u_t$ be the first eigenfunction with respect to $\mathcal L_{g_t}$ with
$$
\int_{M}u_t^2\,\mathrm d\mu_{g_t}=1.
$$
Clearly we have $u_0=\vol(M,g)^{-1/2}$. It follows from \cite[P. 423-426]{Kato1995} that $\lambda_{g_t}$ and $u_t$ are analytic with respect to $t$. Then we can compute
\begin{equation}\label{Eq: deformation}
\begin{split}
\left.\frac{\mathrm d}{\mathrm dt}\right|_{t=0}\lambda_{g_t}&=\left.\frac{\mathrm d}{\mathrm dt}\right|_{t=0}\int_M|\nabla_{g_t}u_t|^2+\frac{R(g_t)-R(g)}{2}u_t^2\,\mathrm d\mu_{g_t}\\
&=\vol(M,g)^{-1}\int_M\langle \Ric(g),\zeta\rangle_g\,\mathrm d\mu_g.
\end{split}
\end{equation}

In the following, we deal with two subcases.

{\it Case 2a. $\Ric(v,v)<0$ for some unit vector $v\in TM$.} Without loss of generality we can assume that $v$ is an eigenvector of $\Ric(g)$. Denote $\omega$ to be the dual one-form with respect to $v$. Then it is clear that
$$
\langle\Ric(v,v)\omega\otimes\omega,\Ric(g) \rangle_g>0.
$$
From continuity we can extend $\Ric(v,v)\omega\otimes\omega$ to a global non-positive $(0,2)$-tensor $\zeta$ on $M$ such that $\langle\zeta,\Ric(g) \rangle_g\geq 0$. From \eqref{Eq: deformation} we see $\lambda_{g_t}$ is positive for small positive $t$. As a result, $(M,g_t)$ has $T^1$-stablized scalar curvature lower bound $R(g)+2\lambda_{g_t}$. Moreover, we have $g_t\geq g$ as quadratic forms.

Next we just deal with the case when $n=5$ and the rest cases are the same if we consider the product manifold $M^n\times T^{5-n}$ instead. Notice that the universal covering $(\hat M,\hat g_t)$ of $(M,g_t)$ is non-compact due to the $H_k$-vanishing condition for $k\geq 3$. It follows from Lemma \ref{Lem: geometry of cover} that  $(\hat M,\hat g_t)$ satisfies all the hypothesis of Proposition \ref{Prop: selection}.
Fix a small $\tilde\epsilon$ such that
$$
\left(1-\frac{2\lambda_{g_t}}{\max_MR(g)+2\lambda_{g_t}}\right)(8\pi+\tilde\epsilon)<8\pi.
$$
From Proposition \ref{Prop: selection} we can find a piecewise smooth homotopically non-trivial $2$-sphere $S_{\tilde\epsilon}$ in $\hat M$ such that
$$
\int_{S_{\tilde\epsilon}} R(\hat g)+2\lambda_{g_t}\,\mathrm d\mu_{\hat g_t}\leq 8\pi+\tilde\epsilon.
$$
Take $S$ to be the projection of $S_{\tilde\epsilon}$ in $M$. Clearly we have
\[
\begin{split}
\int_{S} R(g)\,\mathrm d\mu_{g}&
=\int_{S_{\tilde\epsilon}} R(\hat g)\,\mathrm d\mu_{\hat g}\\
&\leq\left(1-\frac{2\lambda_{g_t}}{\max_MR(g)+2\lambda_{g_t}}\right)\int_{S_{\tilde\epsilon}} R(\hat g)+2\lambda_{g_t}\,\mathrm d\mu_{\hat g}\\
&\leq\left(1-\frac{2\lambda_{g_t}}{\max_MR(g)+2\lambda_{g_t}}\right)\int_{S_{\tilde\epsilon}} R(\hat g)+2\lambda_{g_t}\,\mathrm d\mu_{\hat g_t}\\
&\leq \left(1-\frac{2\lambda_{g_t}}{\max_MR(g)+2\lambda_{g_t}}\right)(8\pi+\tilde\epsilon)<8\pi.
\end{split}
\]
In this case we have $Q_{GB}(M,g)<8\pi$.

{\it Case 2b. $\Ric(g)\geq 0$ but $\Ric(g)\nequiv 0$.} With a similar argument as in Case 1, we see that the universal covering $(\hat M,\hat g)$ splits as $(\mathbf S^2,\hat h)\times \mathbf R^{n-2}$, where $\hat h$ is a smooth metric on $\mathbf S^2$ with non-negative sectional curvature. Clearly $\pi_2(M)\neq 0$ and after taking $S$ to be the projection of $\mathbf S^2$ we have
$$
\int_{S} R(g)\,\mathrm d\mu_{g}=8\pi.
$$
On the other hand, with the projection map $\hat M\to \mathbf S^2$ it is easy to see that for any homotopically non-trivial $2$-sphere $S$ it holds
$$
\int_{S} R(g)\,\mathrm d\mu_{g}\geq 8\pi.
$$
In this case we have $Q_{GB}(M,g)=8\pi$.

Now we make a conclusion from above argument. The discussion of Case 1 and Case 2 tells us that $(M,g)$ is either flat or $\pi_2(M)\neq 0$. In the latter case, we have $Q_{GB}(M,g)\leq 8\pi$ and moreover the equality holds if and only if the universal covering $(\hat M,\hat g)$ splits as $(\mathbf S^2,\hat h)\times \mathbf R^{n-2}$, where $(\mathbf S^2,\hat h)$ has non-negative sectional curvature.
\end{proof}

\subsection{Proof of Theorem \ref{Thm: main dominated} and Corollary \ref{Cor: systole}}
As a preparation, we modify Proposition \ref{Prop: selection} to the following version.
\begin{proposition}\label{Prop: selection dominated}
Let $(\hat M_0^5,\hat g_0)$ be a complete Riemannian manifold with  $C^0$-bounded geometry, uniform filling property and $H_k(\hat M,\mathbf Z)=0$ for all $k\geq 3$ and it contains a geodesic line $\hat\gamma$. If $(\hat M,\hat g)$ is a complete Riemannian manifold with $C^0$-bounded geometry, which admits a proper globally Lipschitz map $\hat f:\hat M\to \hat M_0$ with non-zero degree and has $T^l$-stablized scalar curvature lower bound $\sigma\geq \sigma_0>0$ for some constant $\sigma_0$, then for any $\tilde\epsilon>0$ we can find a homotopically non-trivial piecewise smooth $2$-sphere $S_\epsilon$ in $\hat M$ such that
$$
\int_{S_{\tilde\epsilon}}\sigma|_{S_{\tilde \epsilon}}\,\mathrm d\mu_{\hat g}\leq 8\pi+\tilde\epsilon.
$$
In particular we have $\pi_2(\hat M)\neq 0$.
\end{proposition}
\begin{proof}
As before we take smooth functions $\hat\rho_1$ and $\hat \rho_2$ such that
$$
|\hat\rho_1(\cdot)-\dist(\hat\gamma(0),\cdot)|\leq 1
$$
and
$$
|\hat \rho_2(\cdot)-\dist(\hat\gamma([0,+\infty)),\cdot)|\leq 1.
$$
Let $L$ be a large positive constant to be determined later. We take $c_2$ to be a regular value of $\hat\rho_2$ contained in $(L+3,L+4)$ and $c_1$ to be a regular value of $\hat\rho_1$ contained in $(3L+13,3L+14)$. It follows from Lemma \ref{Lem: extrinsic cutting} and Lemma \ref{Lem: filling radius} that we can find an embedded closed 3-submanifold $S_0\subset \hat M_0$ homologous to zero in $H_3(\hat M,\mathbf Z)$ such that
$$\fillrad(\deg \hat f\cdot S_0,\hat M_0)\geq \min\{c_1-2c_2-3,c_2-1\}\geq L+2.$$
From our construction the submanifold $S_0$ is the boundary of the smooth hypersurface $\Sigma_0:=\hat\rho_1^{-1}((-\infty,c_1])\cap \hat\rho_2^{-1}(c_2)$. From differential topology we can pick up a smooth map $\hat f_1$ homotopic to $\hat f$ such that $|\hat f_1-\hat f|\leq 1$ and $\hat f_1$ is transverse to $S_0$ and $\Sigma_0$.
Denote $S=\hat f_1^{-1}(S_0)$. Then $S$ is homologous to zero since we have $\hat f_1^{-1}(S_0)=\partial \hat f_1^{-1}(\Sigma_0)$. 
From Lemma \ref{Lem: pasting} we can find a smooth embedded hypersurface $\Sigma$ with $\partial\Sigma=S$ and $T^{l+1}$-stablized scalar curvature lower bound $\sigma\geq \sigma_0>0$.

Let $r_0$ be another positive constant to be determined later. It follows from Lemma \ref{Lem: intrinsic cutting} that we can find an embedded hypersurface $S_{r_0}$ homologous to $S$ in $\Sigma$ such that $S_{r_0}$ is contained in $B_{r_0}(S)$ and it has $T^{l+2}$-stablized scalar curvature lower bound
\begin{equation*}
\sigma|_{S_{r_0}}-\frac{4(l+4)\pi^2}{(l+5)r_0^2}.
\end{equation*}
Take
$$
\epsilon=\frac{4\pi^2}{(l+5)r_0^2}.
$$
 It follows from Proposition \ref{Prop: slice and dice} that we can find finitely many embedded $2$-spheres $S_1$, ..., $S_p$ and embedded $2$-disks $D_1$, ..., $D_q$ in $S_{r_0}$ such that
\begin{itemize}
\item the diameters of all spheres $S_i$ and disks $D_j$ with respect to induced metric are bounded from above by
$$
\Lambda_1:=2\pi\left(\sigma_0-\frac{4\pi^2}{r_0^2}\right)
;$$
\item we have
$$
\int_{S_i}\left(\sigma|_{S_i}-\frac{4\pi^2}{r_0^2}\right)\,\mathrm d\mu\leq 8\pi\mbox{ and } \int_{D_j}\left(\sigma|_{D_j}-\frac{4\pi^2}{r_0^2}\right)\,\mathrm d\mu\leq 4\pi;
$$
\item spheres $S_i$ are pairwise disjoint, disks $D_j$ are also pairwise disjoint but intersect one of $S_i$ transversely on its boundary $\partial D_j$;
\item the diameter of each component of the complement
\begin{equation*}
\hat S:=S_{r_0}-\left(\bigcup_{i=1}^k S_i\right)\cup\left(\bigcup_{j=1}^p D_j\right)
\end{equation*}
is no greater than
$$
\Lambda_2:=2\pi\left(\sigma_0-\frac{4\pi^2}{r_0^2}\right)+4r_0,
$$
where the diameter is computed with respect to the distance in $S_{r_0}$.
\end{itemize}
Denote $U_1,\ldots, U_m$ to be the components of the complement $\hat S$. From the selection argument we have the following alternative:
\begin{itemize}
\item either there is a homotopically non-trivial $2$-sphere $\tilde S$ in $\hat M$ such that
$$
\int_{\tilde S}\sigma|_{\tilde S}\,\mathrm d\mu_{\hat g}\leq 8\pi\left(1-\frac{4\pi^2}{\sigma_0r_0^2}\right)^{-1},
$$
\item or the boundary components $(\partial U_k)_\tau$ of region $U_k$ can be filled in $\hat M$.
\end{itemize}
In the latter case, through a similar argument as in the proof of Theorem \ref{Thm: main} we have
$$
L\leq r_0\Lip \hat f+F\left(\Lambda_2\Lip\hat f+F(3\Lambda_1\Lip\hat f)\right).
$$
Given any $\tilde\epsilon>0$ the proof is now completed by first taking $r_0$ large enough such that
$$
8\pi\left(1-\frac{4\pi^2}{\sigma_0r_0^2}\right)^{-1}\leq 8\pi+\tilde\epsilon
$$
and then taking
$
L
$ large enough.
\end{proof}

Now we are ready to prove Theorem \ref{Thm: main dominated}.
\begin{proof}[Proof of Theorem \ref{Thm: main dominated}]
Pick up an arbitrary smooth metric $g_0$ on $M_0$. Let $(\hat M_0,\hat g_0)$ be the universal covering of $(M_0,g_0)$ and let $p:(\hat M,\hat x)\to (M,x)$ be the covering space of $M$ such that $p_*(\pi_1(\hat M,\hat x))=\ker f_*$. After lifting the map $f$ to $\hat f:\hat M\to \hat M_0$ we have the following commutative diagram
\begin{equation*}
\xymatrix{
  \hat M \ar[d]_{p} \ar[r]^{\hat f}
                & \hat M_0 \ar[d]^{p_0}  \\
  M \ar[r]^{f}
                & M_0           .}
\end{equation*}
It follows from \cite[Lemma 18]{CLL2021} that the map $\hat f$ is proper and we have
$$\deg \hat f=\deg f\neq 0\mbox{ and } \Lip\hat f=\Lip f<+\infty.$$

Now we make a discussion similar as before.

{\it Case 1. $\Ric(g)\equiv 0$.} It follows from \cite[Theorem 3]{CG1971} that $(M,g)$ has a finite covering space $(\tilde M,\tilde g)$ with fundamental group $\mathbf Z^k$, where $k$ is the dimension of the $\mathbf R^k$-component in the splitting of the universal covering of $M$. From the Ricci-flatness all we need to show is $k\geq n-3$. This is obvious when $n\leq 3$ and so we deal with the case when $n=4,5$. Notice that $M_0$ has a finite covering space $\tilde M_0$ whose fundamental group is a quotient of $\mathbf Z^k$. Through further lifting we can assume that $\pi_1(\tilde M_0)$ is free and so isometric to $\mathbf Z^{k_0}$ for some $k_0\leq k$. From the $H_i$-vanishing condition for $i\geq 3$ of $\hat M_0$, we see that $\hat M_0$ is non-compact and so $k\geq k_0\geq 1$. This handles the case when $n=4$. For $n=5$ we take the following contradiction argument. Suppose that we have $k_0= 1$. Let $\gamma$ be a closed curve generating $\pi_1(\tilde M_0)\cong \mathbf Z$. From the Poincar\'e duality we can find an embedded two-sided hypersurface $\tilde\Sigma_0$ with non-zero intersection number with $\gamma$, which represents a non-zero homology class. Notice that the map $i_*:\pi_1(\tilde\Sigma_0)\to \pi_1(\tilde M_0)$ must be zero. Otherwise there is a closed curve $\gamma'$ on $\tilde\Sigma_0$ homotopic to $c\gamma$ for some $c\neq 0\in\mathbf Z$. This is impossible since it further implies that the intersection number of $\gamma$ and $\tilde\Sigma_0$ is zero. As a result, $\tilde\Sigma_0$ can be lifted to an embedded hypersurface $\hat\Sigma_0$ in $\hat M_0$. From the $H_i$-vanishing condition for $i\geq 3$ we see that $\hat \Sigma_0$ is null homologous in $\hat M_0$ and so is $\tilde \Sigma_0$ in $\tilde M_0$. This leads to a contradiction and so we have $k\geq k_0\geq 2$.

{\it Case 2a. $\Ric(v,v)<0$ for some unit vector $v\in TM$.} Through the same deformation argument in the proof of Theorem \ref{Thm: main}, we can find a smooth metric $g_t$ on $M$ such that $g_t\geq g$ and $(M,g_t)$ has $T^1$-stablized scalar curvature lower bound $R(g)+2\lambda_t$ for some positive constant $\lambda_t$. Denote $\hat g_t=p^*g_t$. Next we only consider the case when $n=5$, otherwise we consider $(M\times T^{5-n},g_t+\sum_i\mathrm d\theta_i^2)$ instead. Notice that the manifolds $(\hat M_0,\hat g_0)$ and $(\hat M,\hat g_t)$ satisfies the hypothesis of Proposition \ref{Prop: selection dominated} and so for any $\tilde\epsilon>0$ we can find a homotopically non-trivial $2$-sphere $S_{\tilde\epsilon}$ in $\hat M$ such that
$$
\int_{S_{\tilde\epsilon}}R(\hat g_t)+2\lambda_t\,\mathrm d\mu_{\hat g_t}\leq 8\pi+\tilde\epsilon.
$$
In particular, we have $\pi_2(M)\neq 0$. Let $S$ be the projection of $S_{\tilde\epsilon}$ in $M$. Just as in the proof of Theorem \ref{Thm: main}, by taking $\tilde\epsilon$ sufficiently small can show
$$
\int_S R(g)\,\mathrm d\mu_g<8\pi,
$$
which yields $Q_{GB}(M,g)<8\pi$.

{\it Case 2b. $\Ric(g)\geq 0$ but $\Ric(g)\nequiv 0$.} From the discussion in Case 1 we know that the universal covering $(\hat M,\hat g)$ of $(M,g)$ splits into product manifold $(\hat N^{n-k},\hat h)\times \mathbf R^k$ with $k\geq n-3$, where $(\hat N^{n-k},\hat h)$ is a simply connected closed manifold. Notice that $(M,g)$ admits a smooth metric with positive scalar curvature. From Proposition \ref{Prop: selection dominated} we see $\pi_2(M)=\pi_2(\hat M)\neq 0$ and then the Poincar\'e duality applied to $\hat N$ rules out the possibility that $n-k= 3$. This means that $\hat N$ is a $2$-sphere and so the projection $S$ of $\hat N$ in $M$ satisfies
$$
\int_S R(g)\,\mathrm d\mu_g=8\pi.
$$
As before, with the projection map $\hat M\to \hat N$ it is easy to see that for any homotopically non-trivial $2$-sphere $S$ in $M$ it holds
$$
\int_{S} R(g)\,\mathrm d\mu_{g}\geq 8\pi,
$$
and so we have $Q_{GB}(M,g)=8\pi$.
\end{proof}

Finally let us prove Corollary \ref{Cor: systole}.
\begin{proof}[Proof of Corollary \ref{Cor: systole}]
Since $(M,g)$ has positive scalar curvature, we have $\pi_2(M)\neq 0$ and
$$
\inf_MR(g)\cdot\sys_2(M,g)\leq Q_{GB}(M,g)\leq 8\pi.
$$
When the equality holds, we see that $R(g)$ is a positive constant function and $Q_{GB}(M,g)=8\pi$. As a result, the universal covering splits as $\mathbf S^2(1)\times \mathbf R^{n-2}$ up to rescaling.
\end{proof}

\noindent{\bf Conflict of interest statement} On behalf of all authors, the corresponding author states that there is no conflict of interest.

\bibliography{bib}
\bibliographystyle{alpha}
\end{document}